\documentclass[11pt, a4paper, twoside]{article}

\usepackage[
  top=1.25in, bottom=1.25in,
  left=1.25in, right=1.25in
]{geometry}

\usepackage[T1]{fontenc}
\usepackage[sc,osf]{mathpazo}
\usepackage{eulervm}
\usepackage[scaled=0.85]{inconsolata}

\usepackage[final]{microtype}

\usepackage{amsmath, amssymb, amsthm}
\usepackage{mathtools}
\usepackage{bm}

\usepackage[dvipsnames]{xcolor}
\usepackage[
  colorlinks = true,
  linkcolor  = RoyalBlue,
  citecolor  = ForestGreen,
  urlcolor   = RoyalBlue
]{hyperref}

\usepackage{aliascnt}

\newtheorem{theorem}{Theorem}[section]

\newcommand{\makealiasenv}[2]{%
  \newaliascnt{#1}{theorem}%
  \newtheorem{#1}[#1]{#2}%
  \aliascntresetthe{#1}%
}

\makealiasenv{lemma}{Lemma}
\makealiasenv{proposition}{Proposition}
\makealiasenv{corollary}{Corollary}
\makealiasenv{conjecture}{Conjecture}
\makealiasenv{claim}{Claim}

\theoremstyle{definition}
\makealiasenv{definition}{Definition}
\makealiasenv{example}{Example}
\makealiasenv{exercise}{Exercise}
\makealiasenv{problem}{Problem}
\makealiasenv{construction}{Construction}

\theoremstyle{remark}
\makealiasenv{remark}{Remark}
\makealiasenv{notation}{Notation}
\makealiasenv{observation}{Observation}

\newtheorem*{theorem*}{Theorem}
\newtheorem*{lemma*}{Lemma}
\newtheorem*{remark*}{Remark}

\usepackage[capitalise, noabbrev]{cleveref}

\crefname{theorem}{Theorem}{Theorems}
\crefname{lemma}{Lemma}{Lemmas}
\crefname{proposition}{Proposition}{Propositions}
\crefname{corollary}{Corollary}{Corollaries}
\crefname{conjecture}{Conjecture}{Conjectures}
\crefname{claim}{Claim}{Claims}
\crefname{definition}{Definition}{Definitions}
\crefname{example}{Example}{Examples}
\crefname{exercise}{Exercise}{Exercises}
\crefname{remark}{Remark}{Remarks}
\crefname{equation}{Equation}{Equations}
\crefname{section}{Section}{Sections}
\crefname{figure}{Figure}{Figures}
\crefname{table}{Table}{Tables}

\usepackage[square, numbers, sort&compress]{natbib}
\usepackage{graphicx}
\usepackage{booktabs}
\usepackage{enumitem}
\usepackage{csquotes}
\usepackage{stmaryrd}
\usepackage{authblk}
\usepackage{tikz}
\usepackage{fancyhdr}
\usepackage{titlesec}
\titleformat{\section}{\large\bfseries\scshape}{\thesection.}{0.6em}{}
\titleformat{\subsection}{\normalsize\bfseries}{\thesubsection.}{0.5em}{}

\usepackage{abstract}

\newcommand{\R}{\mathbb{R}}

\newcommand*\de{\partial}

\newcommand{\esssup}{\operatorname*{ess\,sup}}

\begin{document}

\title{\textbf{\Large LIPSCHITZ REGULARITY FOR PARABOLIC FRACTIONAL $p$-LAPLACE EQUATIONS}}

\author[]{HARSH PRASAD}
\affil[]{Fakult\"{a}t f\"{u}r Mathematik, Universit\"{a}t Bielefeld\\
  \texttt{hprasad@math.uni-bielefeld.de}}

\date{\today}
\maketitle
\pagestyle{fancy}
\fancyhf{}
\fancyhead[C]{%
  \ifodd\value{page}%
    LIPSCHITZ REGULARITY %
  \else%
    PRASAD%
  \fi%
}
\fancyfoot[C]{\thepage}
\renewcommand{\headrulewidth}{0pt}
\begin{abstract}
We prove that local weak solutions to nonlocal parabolic
$p$-Laplace equations are locally Lipschitz continuous in space,
uniformly in time for every $1<p<\infty$ and $s \in (0,1)$ whenever $sp > p-1$. Our results hold for
symmetric, translation-invariant kernels satisfying standard
ellipticity bounds, including kernels that may be discontinuous
and require only that the tail of the solution be bounded. In the linear case, our proof provides a different route avoiding blow up arguments and Liouville theorems. 
\end{abstract}



\tableofcontents

\section{Introduction}
\label{sec:intro}
We are interested in studying gradient regularity for weak solutions to the following non-linear, nonlocal parabolic equation:
\begin{equation}\label{eq:main-eq}
  \partial_t u + L_K u = 0 \quad \text{in } \Omega \times I,
\end{equation}
where $\Omega \subset \R^N$ is open, $I \subset \R$ is an open
interval, and
\[
  L_K u(x,t) := \mathrm{P.V.} \int_{\R^N}
  |u(x,t)-u(y,t)|^{p-2}(u(x,t)-u(y,t))\, K(x-y)\, dy.
\]
The kernel $K : \R^N \setminus \{0\} \to (0,\infty)$ is
symmetric and satisfies
\begin{equation}\label{eq:kernel-bounds}
  \frac{\lambda}{|z|^{N+sp}} \le K(z)
  \le \frac{\Lambda}{|z|^{N+sp}}
\end{equation}
for fixed constants $0 < \lambda \le \Lambda < \infty$,
$p \in (1,\infty)$ and $s \in (0,1)$. We do \emph{not} assume
$K$ to be continuous. The model case is the fractional
$p$-Laplace equation
\[
  \partial_t u + (-\Delta_p)^s u = 0,
\]
which corresponds to $K(z) = |z|^{-(N+sp)}$. 

\subsection{Background and Novelty}

Local regularity for nonlocal, nonlinear equations has been
intensively studied over the past decade. Local boundedness and
H\"{o}lder continuity for equations with kernels
satisfying~\eqref{eq:kernel-bounds} were established in the
elliptic case in~\cite{DiCastroKuusiPalatucci2016, Cozzi2017,
AdimurthiPrasadTewary2023} and in the parabolic case
in~\cite{AdimurthiPrasadTewary2025LocalHolder,
Liao2024HolderParabolicFracPLap, DingZhangZhou2021,
PrasadTewary2023}. These results hold for
the full class of kernels satisfying~\eqref{eq:kernel-bounds} including kernels that are not translation invariant. However, as is typically the case with  DeGiorgi-Nash-Moser techniques, they yield only a
small, non-explicit H\"{o}lder exponent. It is natural to ask if we can get sharp exponents under less general assumptions on the kernel. 

For the fractional $p$-Laplacian, sharp higher H\"{o}lder regularity in space was
obtained in~\cite{BrascoLindgrenSchikorra2018} for $p \ge 2$
and in~\cite{GarainLindgren2024} for $1 < p < 2$, with the
optimal spatial exponent
$\gamma^\circ = \min\{1, sp/(p-1)\}$. In particular, when
$sp > p-1$, solutions were shown to be \emph{almost} Lipschitz. The
parabolic analogues, including higher H\"{o}lder regularity in
both space and time, appear
in~\cite{BrascoLindgrenStromqvist2021,
GarainLindgrenTavakoli2025}.

Whether solutions are actually Lipschitz continuous in space
when $sp > p-1$ remained open in both the elliptic and
parabolic settings. In the elliptic case, this was recently
settled by Biswas--Topp~\cite{BiswasTopp2025}, who introduced
the Ishii--Lions method~\cite{IshiiLions1990} into the nonlocal
nonlinear context. These methods were extended to yield better regularity for fractional $p-$harmonic functions in \cite{BiswasSen2025} and subsequently, $C^{1,\alpha}$ regularity for fractional $p-$harmonic functions  was
 established
in~\cite{GiovagnoliJesusSilvestre2025}. Earlier applications
of Ishii--Lions methods to nonlocal
equations~\cite{BarlesChasseigneImbert2011,
BarlesChasseigneCiomagaImbert2012, BiswasQuaasTopp2025} relied
on a sub-additive (Pucci-type) structure that leads to useful
cancellations; this structure is unavailable for the fractional
$p$-Laplacian, and~\cite{BiswasTopp2025} handles this
difficulty through a careful treatment of the $I_3$ integral (see \cref{lem:parabolic-I3-sub} and \cref{lem:parabolic-I3-super}.)

Very recently, and independently of the present work,
Jesus--Sobral--Urbano~\cite{JesusSobralUrbano2026} established
spatial Lipschitz regularity for weak solutions to the
parabolic fractional $p$-Laplace equation in the degenerate
range $p \ge 2$, as well as time regularity via a
barrier argument. Our results overlap with theirs in this range
for the explicit kernel $K(z) = |z|^{-(N+sp)}$, but differ in
several significant respects.

\begin{enumerate}[label=(\roman*)]
\item \emph{Subquadratic case.} We treat the full range
  $p > 1$, including $1 < p < 2$. The subquadratic case is not a simple extension
  of the superquadratic case since in the latter case we can evaluate the operator on smooth
  functions whereas in the former case we cannot do so without running into singularities. This is the main reason why in \cite{KorvenpaaKuusiLindgren2019}, one needs to work with  an auxilliary class of test functions in the definition of viscosity solutions. In particular, extending \cite{JesusSobralUrbano2026} directly to the subquadratic case would need the introduction of such special class of test functions in the parabolic setup. On the other hand, we circumvent the issue entirely. 

\item \emph{Discontinuous kernels.} We work with a more general class of kernels including discontinuous kernels. Doing so brings it in line with known results in the elliptic case where spatial Lipschitz 
  regularity is known for discontinuous, translation invariant, symmetric kernels \cite{RosOtonSerra2016RegularityGeneralStable, FernandezRealRosOton2017RegularityParabolic}. Whereas in the linear case, such a result proceeds via heat kernel estimates, Liouville theorems and blow up arguments \cite[Theorem 2.4.3]{fernandezreal2024integrodifferentialellipticequations}, our proof is different. 

\item \emph{Weaker tail assumption.} We impose only
  $u \in L^\infty_{\mathrm{loc}}(I;\, L^{p-1}_{sp}(\R^N))$,
  that is, the tail of $u$ is merely bounded in time.
  The paper~\cite{JesusSobralUrbano2026} requires the
  strictly stronger condition
  $u \in C_{\mathrm{loc}}(I;\, L^{p-1}_{sp}(\R^d))$,
  i.e.\ continuity of the tail in time. Such an assumption also occurs when studying linear 
  equations \cite{LaraDavila}. We bypass the requirement. 
  
\item \emph{A parametrised family of Ishii-Lions functionals} The
  paper~\cite{JesusSobralUrbano2026} relies on the the Jensen--Ishii
  lemma~\cite[Theorem 8.3]{CrandallIshiiLions1992} to produce limiting
  jets. We eschew the usage of the 
  viscosity machinery by working with a parametrised family of Ishii-Lions functionals (cf.~\cref{sec:parabolic-ishii-lions})
\end{enumerate}

We also note that~\cite{JesusSobralUrbano2026} establishes a
full viscosity solution framework for the parabolic fractional
$p$-Laplacian when $p>2$, including the equivalence between weak and
viscosity solutions and a comparison principle for viscosity
solutions. 

\subsection{Main Theorem}

For $(x_0, t_0) \in \R^N \times \R$ and $R, S > 0$ we write
$Q_{R,S}(x_0,t_0) := B_R(x_0) \times (t_0 - S, t_0]$.
Given a weak solution $u$ in $Q_{2R, 2R^{sp}}(x_0, t_0)$,
define
\begin{equation}\label{eq:MR-def}
  M_R := \|u\|_{L^\infty(Q_{2R,2R^{sp}}(x_0,t_0))}
       + \mathrm{Tail}_\infty(u;\, x_0, R,\,
         (t_0-2R^{sp}, t_0])
       + 1.
\end{equation}

\begin{theorem}[Spatial Lipschitz regularity]
\label{thm:main}
   Let $u$ be a local weak solution to~\eqref{eq:main-eq}. Suppose $(x_0, t_0) \in \Omega \times I$
   and $R > 0$ are such that
  $Q_{2R,2R^{sp}}(x_0,t_0) \Subset \Omega \times I$.
  
  Then $u(\cdot, t)$ is locally Lipschitz in space,
  uniformly in time and  with $M_R$ as
  in~\eqref{eq:MR-def}, the following estimates hold for
  all $t \in (t_0 - R^{sp}/2,\, t_0]$:
  \begin{align}
    \sup_{t}\,[u(\cdot,t)]_{C^{0,1}(B_{R/2}(x_0))}
    &\le C\, M_R\, R^{-1}
    &&\text{if } p \ge 2, \label{eq:lip-pge2}\\
    \sup_{t}\,[u(\cdot,t)]_{C^{0,1}(B_{R/2}(x_0))}
    &\le C\, M_R^{1+\frac{2-p}{sp}}\, R^{-1}
    &&\text{if } 1 < p < 2. \label{eq:lip-plt2}
  \end{align}
  The constant $C > 0$ depends only on $N, s, p,
  \lambda, \Lambda$.
\end{theorem}

The proof is given in Section~\ref{sec:proofs}. It combines
the touching lemma
(Proposition~\ref{prop:weak-to-visc-parabolic-noncritical}),
the scaling reduction
(Section~\ref{sec:scaling-reduction}), and the Ishii--Lions
argument (Section~\ref{sec:parabolic-ishii-lions}).

\begin{remark}[Sharpness of the condition]
\label{rem:sharpness}
  Our result extends to equations with a bounded source term (cf.~\cref{sec:extensions}) and thus the condition $sp>p-1$ is sharp for the method ~\cite[Example~1.6]{BrascoLindgrenSchikorra2018}. However, as in \cite{BiswasSen2025}, it may be possible to extend the range of exponents if we assume higher regularity on the kernel.
\end{remark}

\begin{remark}[Time regularity]\label{rem:time-regularity}
  The spatial Lipschitz estimate of Theorem~\ref{thm:main}
  can be used to improve time regularity (for more general kernels) via a barrier
  argument; we do not pursue it here. 
\end{remark}

\begin{remark}[Discontinuous kernels]
.  All prior higher regularity results
  (\cite{BrascoLindgrenSchikorra2018, GarainLindgren2024,BrascoLindgrenStromqvist2021,
  GarainLindgrenTavakoli2025}) worked with the explicit
  kernel $K(z) = |z|^{-(N+sp)}$.
  Theorem~\ref{thm:main} holds for any $K$
  satisfying~\eqref{eq:kernel-bounds}, including
  discontinuous ones. Our method extends to proving higher H\"older regularity and thus generalises these previous works (cf.~\cref{sec:extensions}).
\end{remark}

\subsection{Extensions}\label{sec:extensions-intro}

The method extends to a broader class of operators. In
Section~\ref{sec:extensions} we state (without proofs)
Lipschitz regularity results for the parabolic fractional
$(p,q)$-Laplacian
$\de_t u + (-\Delta_p)^{s_1}u + (-\Delta_q)^{s_2}u$ and for parabolic nonlocal
double-phase operators
$\de_t + (-\Delta_p)^{s_1} + \xi(x)(-\Delta_q)^{s_2}$, under
analogues of the condition $sp' > 1$. The arguments are
parallel to those of~\cite{BiswasTopp2025} combined with a suitable version of the
parabolic touching lemma (cf.~\cref{sec:touching}).

\subsection{Outline}

Section~\ref{sec:prelim} collects notation, function spaces,
and known regularity results used throughout.
Section~\ref{sec:touching} proves the touching lemma for
weak solutions.
Section~\ref{sec:scaling-reduction} records scaling
reductions. Section~\ref{sec:proofs} contains the main
proof: the Ishii--Lions framework, the integral estimates,
and the H\"{o}lder bootstrap culminating in
Theorem~\ref{thm:main}. Section~\ref{sec:extensions} states
the extensions to related operators.



\section{Preliminaries}
\label{sec:prelim}

\subsection*{Notation}

We write $B_R(x_0) := \{x \in \R^N : |x - x_0| < R\}$ and
$B_R := B_R(0)$. For $(x_0, t_0) \in \R^N \times \R$ and
$R, S > 0$ we set
\[
  Q_{R,S}(x_0,t_0)
  := B_R(x_0) \times (t_0 - S,\, t_0], \qquad
  Q_{R,S} := Q_{R,S}(0,0).
\]
The conjugate exponent of $p \in (1,\infty)$ is
$p' := p/(p-1)$. We write $J_p(t) := |t|^{p-2}t$ for
$t \in \R$.

\subsection*{Function spaces}

\paragraph{Fractional Sobolev spaces.}
Let $E \subset \R^N$ be an open set,  $1<p<\infty$ and $s \in (0,1)$. The fractional Sobolev space
$W^{s,p}(E)$ consists of all $u \in L^p(E)$ such that
\[
  [u]_{W^{s,p}(E)}^p
  := \int_E \int_E
  \frac{|u(x)-u(y)|^p}{|x-y|^{N+sp}}\, dx\, dy < \infty.
\]
It is a reflexive Banach space under the norm $\|u\|_{W^{s,p}(E)} :=
\|u\|_{L^p(E)} + [u]_{W^{s,p}(E)}$. The local
space $W^{s,p}_{\mathrm{loc}}(E)$ consists of all $u$
with $u \in W^{s,p}(U)$ for every open $U \Subset E$.

\paragraph{Tail space.}
The tail space $L^{p-1}_{sp}(\R^N)$ is defined by
\[
  L^{p-1}_{sp}(\R^N)
  := \Bigl\{ u \in L^{p-1}_{\mathrm{loc}}(\R^N) :
     \int_{\R^N}
     \frac{|u(x)|^{p-1}}{1+|x|^{N+sp}}\, dx < \infty
     \Bigr\}.
\]
For $u(\cdot,t) \in L^{p-1}_{sp}(\R^N)$, $x_0 \in \R^N$
and $r > 0$, we set
\[
  \mathrm{Tail}(u(\cdot,t);\, x_0, r)
  := \Bigl( r^{sp}
     \int_{\R^N \setminus B_r(x_0)}
     \frac{|u(x,t)|^{p-1}}{|x-x_0|^{N+sp}}\, dx
     \Bigr)^{\!\frac{1}{p-1}}.
\]
For a time interval $I = (t_1, t_2)$ we define the
parabolic tail
\[
  \mathrm{Tail}_\infty(u;\, x_0, r, I)
  := \esssup_{t \in I}\,
     \mathrm{Tail}(u(\cdot,t);\, x_0, r).
\]

\paragraph{H\"{o}lder seminorms.}
For an open set $E \subset \R^N$ and $\alpha \in (0,1]$,
\[
  [u]_{C^{0,\alpha}(E)}
  := \sup_{\substack{x,y \in E \\ x \neq y}}
     \frac{|u(x)-u(y)|}{|x-y|^\alpha}.
\]
The case $\alpha = 1$ is the Lipschitz seminorm.

\subsection*{Definition of weak solution}

\begin{definition}\label{def:weak-solution}
Let $\Omega \subset \R^N$ be open, $I = (t_1, t_2)$,
$1<p<\infty$, $s \in (0,1)$, and let $K$
satisfy~\eqref{eq:kernel-bounds}. A function
\[
  u \in L^p_{\mathrm{loc}}(I;\,
    W^{s,p}_{\mathrm{loc}}(\Omega))
  \cap L^\infty_{\mathrm{loc}}(I;\,
    L^{p-1}_{sp}(\R^N))
  \cap C_{\mathrm{loc}}(I;\,
    L^2_{\mathrm{loc}}(\Omega))
\]
is a \emph{local weak solution} to \eqref{eq:main-eq} if, for
every $\varphi \in C^\infty_c(\Omega \times I)$,
\begin{equation}\label{eq:weak-form}
  -\int_I \int_\Omega u\, \partial_t \varphi\, dx\, dt
  + \frac{1}{2}\int_I \int_{\R^N} \int_{\R^N}
    J_p(u(x,t)-u(y,t))\,
    (\varphi(x,t)-\varphi(y,t))\,
    K(x-y)\, dx\, dy\, dt = 0.
\end{equation}
Local weak supersolutions and subsolutions are defined by
replacing $= 0$ with $\ge 0$ and $\le 0$ respectively,
and restricting to $\varphi \ge 0$.
\end{definition}

\subsection*{Known regularity results}

The following results will be used throughout the
paper. They hold under the sole assumption that $K$
satisfies~\eqref{eq:kernel-bounds}; in particular, $K$
need not be continuous.

\begin{proposition}[Local boundedness]
\label{prop:linfinity}
  Let $u$ be a local weak solution
  to~\eqref{eq:main-eq}. Then $u$ is locally bounded.
\end{proposition}
\begin{proof}
  See~\cite{Stromqvist2019LocalBoundedness,
  DingZhangZhou2021, PrasadTewary2023,
  KumagaiWangZhang2024}.
\end{proof}

\begin{proposition}[H\"{o}lder continuity]
\label{prop:c0alpha}
  Let $u$ be a locally bounded, local weak solution
  to~\eqref{eq:main-eq}. Then $u$ is locally H\"{o}lder
  continuous in space and time, with the H\"older norm and the H\"older exponent depending
  only on $N, s, p, \lambda, \Lambda$.
\end{proposition}
\begin{proof}
  See~\cite{AdimurthiPrasadTewary2025LocalHolder,
  Liao2024HolderParabolicFracPLap}.
\end{proof}

In view of Propositions~\ref{prop:linfinity}
and~\ref{prop:c0alpha}, there is no loss of generality in
assuming throughout that $u$ is continuous. 


\section{The Touching Lemma}
\label{sec:touching}

After the scaling reduction of
Section~\ref{sec:scaling-reduction}, we work in the domain
$B_2 \times (-2,0]$ and seek Lipschitz regularity in
$B_1 \times (-1,0]$. We decompose the operator as
\begin{equation}\label{eq:operator-split}
  L_K u = L_K^{B_1} u + L_K^{B_1^c} u,
  \qquad
  L_K^{B_1} u(x,t)
  := \mathrm{P.V.}\int_{B_1(x)}
     J_p(u(x,t)-u(y,t))\,K(x-y)\,dy,
\end{equation}
and set $f(x,t) := - L_K^{B_1^c} u(x,t)$. By the tail
bound $u \in L^\infty_{\mathrm{loc}}(I;\,L^{p-1}_{sp}(\R^N))$
and the normalisation~\eqref{eq:normalization-bound}, the
function $f$ satisfies
\begin{equation}\label{eq:f-bound}
  \|f\|_{L^\infty(B_1 \times (-1,0])} \le C_0,
\end{equation}
where $C_0 > 0$ depends only on $N,s,p,\lambda,\Lambda$
and the normalisation bound. Hence $u$ satisfies
\begin{equation}\label{eq:truncated-eq}
  \partial_t u + L_K^{B_1} u = f
  \quad \text{in } B_1 \times (-1,0],
\end{equation}
with $-C_0 \le f \le C_0$. The main result of this
section is
Proposition~\ref{prop:weak-to-visc-parabolic-noncritical},
which proves the touching lemma for
equation~\eqref{eq:truncated-eq}. The key advantage of
working with $L_K^{B_1}$ rather than $L_K$ is that the
truncated operator integrates only over $B_1(x)$ and thus no tail
condition on the test function $\phi$ is required for the
continuity of the map $(x,t) \mapsto L_K^{B_1}\phi(x,t)$. The bounded
right-hand side $f$ enters the contradiction argument
as a fixed additive constant $\pm C_0$ and does not
affect the divergence argument for large $L$ (cf.~\cref{sec:parabolic-ishii-lions}.)

The argument proceeds in four steps. First
(Lemma~\ref{lem:pv-exists}), we show the principal value
integral of the truncated operator is well-defined and
bounded at points where the gradient does not vanish.
Second
(Lemma~\ref{lem:parabolic-full-operator-continuity}), we
establish continuity of the truncated parabolic operator.
Third
(Lemmas~\ref{lem:parabolic-perturb-down}--\ref{lem:parabolic-L310}),
we show that a smooth strict supersolution of the
truncated equation remains a weak strict supersolution
under small perturbations. Finally, the touching
lemma
(Proposition~\ref{prop:weak-to-visc-parabolic-noncritical})
follows by a contradiction argument (cf.~\cite[Proposition 1.5]{BiswasTopp2025}) using the weak
comparison principle
(Lemma~\ref{lem:parabolic-comparison}).
The results of this section are parabolic analogues of \cite{KorvenpaaKuusiLindgren2019} adapted to our setup.

\begin{lemma}\label{lem:pv-exists}
Let $\Omega \subset \R^N$ be open, $p \in (1,\infty)$,
$s \in (0,1)$. Let $\mathcal{K} \Subset \Omega$ be compact
and $I \subset \R$ a bounded interval. Suppose there exists
an open set $U$ with $\mathcal{K} \Subset U \Subset \Omega$
and a function $\phi \in C^2_x(U \times I)$ with
$\nabla_x\phi(x,t) \neq 0$ for every
$(x,t) \in \mathcal{K} \times I$. Then there exists a
constant $\mathbf{C} > 0$ such that for every
$(x,t) \in \mathcal{K} \times I$ and every
\[
  0 < \delta < \varepsilon <
  \min\bigl\{\mathrm{dist}(\mathcal{K},\partial U),\,
  1\bigr\},
\]
we have
\[
  \left| \int_{B_\varepsilon(x) \setminus B_\delta(x)}
  J_p(\phi(x,t)-\phi(y,t))\, K(x-y)\, dy \right|
  \le \mathbf{C}\,\varepsilon^{p(1-s)}.
\]
The constant $\mathbf{C}$ depends only on $N, p, s,
\Lambda$, $\sup_{U \times I}|D_x^2\phi|$,
$\inf_{\mathcal{K} \times I}|\nabla_x\phi|$, and
$\sup_{\mathcal{K} \times I}|\nabla_x\phi|$; in particular
it is independent of $x \in \mathcal{K}$, $t \in I$, $ \varepsilon>0$ and
$\delta$.
\end{lemma}

\begin{proof}
Fix $x \in \mathcal{K}$, $t \in I$ and
$0 < \delta < \varepsilon <
\mathrm{dist}(\mathcal{K}, \partial U)$.
Set $z = y - x$ and consider the tangent hyperplane
\[
  \ell(y) := \phi(x,t) + \nabla_x\phi(x,t) \cdot (y-x).
\]
Note that $\ell(x) - \ell(x+z) = -\nabla_x\phi(x,t) \cdot z$.
Since $K$ is symmetric and $J_p$ is odd, the integrand
$z \mapsto J_p(-\nabla_x\phi(x,t) \cdot z)K(z)$ is odd,
hence
\[
  \int_{B_\varepsilon \setminus B_\delta}
  J_p(\ell(x) - \ell(x+z))\, K(z)\, dz = 0.
\]
In particular,
\begin{align*}
  &\int_{B_\varepsilon(x) \setminus B_\delta(x)}
  J_p(\phi(x,t) - \phi(y,t))\, K(x-y)\, dy \\
  &\quad = \int_{B_\varepsilon \setminus B_\delta}
     \bigl(J_p(\phi(x,t)-\phi(x+z,t))
     - J_p(\ell(x)-\ell(x+z))\bigr)\, K(z)\, dz.
\end{align*}
Let $a = \phi(x,t) - \phi(x+z,t)$ and
$b = \ell(x) - \ell(x+z)
= -\nabla_x\phi(x,t) \cdot z$. Then, 
$|a - b| \le \tau |z|^2$
where $\tau := \sup_{U \times I}|D_x^2\phi|$. Since
$\nabla_x\phi$ is continuous on $\mathcal{K} \times I$,
there exist $0 < m \le M < \infty$ with
$m \le |\nabla_x\phi(x,t)| \le M$ on
$\mathcal{K} \times I$. Using
$|J_p(a) - J_p(b)|
\le c_p(|a-b| + |b|)^{p-2}|a-b|$ we get:
\begin{itemize}
  \item \emph{for $p \ge 2$:}
    $|J_p(a) - J_p(b)|
    \le C|z|^p$ for $|z| \le 1$.
  \item \emph{for $1 < p < 2$:}
    $|J_p(a) - J_p(b)|
    \le C|\nabla_x\phi(x,t) \cdot z|^{p-2}|z|^2$.
\end{itemize}
In either case, integrating against
$K(z) \le \Lambda|z|^{-N-sp}$ and switching to polar coordinates yields
\[
  \left|\int_{B_\varepsilon \setminus B_\delta}
  (J_p(a) - J_p(b))K(z)\,dz\right|
  \le \mathbf{C}\,\varepsilon^{p(1-s)},
\]
where for $1 < p < 2$ the angular integral
$\int_{S^{N-1}}|e\cdot\omega|^{p-2}\,d\omega$ is finite
since $p > 1$.
\end{proof}

\begin{lemma}\label{lem:parabolic-full-operator-continuity}
Let $\Omega \subset \R^N$ be open, $p \in (1,\infty)$,
$s \in (0,1)$. Let $\mathcal{K} \Subset \Omega$ be compact
and $I \subset \R$ a bounded interval. Suppose there exists
an open set $U$ with $\mathcal{K} \Subset U \Subset \Omega$ and $\rho > 0$ is such that
$B_\rho(x) \subset U$ for all $x \in \mathcal{K}$.
If $\phi \in C^2_x(U \times I) \cap C^1_t(U \times I)$
satisfies $\nabla_x\phi(x,t) \neq 0$ for every
$(x,t) \in \mathcal{K} \times I$, then the map
$(x,t) \mapsto
\partial_t\phi(x,t) + L_K^{B_\rho}\phi(x,t)$
is continuous on $\mathcal{K} \times I$, where
\[
  L_K^{B_\rho}\phi(x,t)
  := \mathrm{P.V.}\int_{B_\rho(x)}
     J_p(\phi(x,t)-\phi(y,t))\,K(x-y)\,dy.
\]
\end{lemma}

\begin{proof}
Let $(x,t),(y,\tau) \in \mathcal{K} \times I$ with
$(x,t) \to (y,\tau)$. Fix
$\varepsilon \in (0, \min\{\rho/2, 1\})$.
Decompose
$L_K^{B_\rho}\phi
= L_K^{B_\varepsilon}\phi
+ L_K^{B_\rho \setminus B_\varepsilon}$.
By Lemma~\ref{lem:pv-exists},
$|L_K^{B_\varepsilon}\phi(x,t)|+|L_K^{B_\varepsilon}\phi(y,\tau)|
\le 2\mathbf{C}\varepsilon^{p(1-s)}$
uniformly. For the annular part, the domain
$B_\rho(x) \setminus B_\varepsilon(x)$ is bounded and
bounded away from the singularity, so the measure
$K(x-z)\,dz$ is finite there. The integral 
\[
\int_{B_\rho(x)\setminus B_{\varepsilon}(x)}
     J_p(\phi(x,t)-\phi(x',t))\,K(x-x')\,dx' = \int_{B_\rho(0)\setminus B_{\varepsilon}(0)}
     J_p(\phi(x,t)-\phi(x+z,t))\,K(z)\,dz
\]
converges
pointwise as $(x,t) \to (y,\tau)$ to
\[
 \int_{B_\rho(0)\setminus B_{\varepsilon}(0)}
     J_p(\phi(y,\tau)-\phi(y+z,\tau))\,K(z)\,dz = \int_{B_\rho(y)\setminus B_{\varepsilon}(y)}
     J_p(\phi(y,\tau)-\phi(y',\tau))\,K(y-y')\,dy'
\]

by continuity of
$\phi$ and dominated convergence; the dominating function being
$C\|\phi\|_{C^1(U \times I)}^{p-1}
|z|^{p-1-N-sp}$,
which is integrable on $B_\rho \setminus B_\varepsilon$. Thus,
\[
\limsup_{(x,t)\to(y,\tau)} |\partial_t\phi(x,t) + L_K^{B_\rho}\phi(x,t) - \partial_t\phi(y,\tau) - L_K^{B_\rho}\phi(y,\tau)| \le 2C\varepsilon^{p(1-s)}.
\]
We now let $\varepsilon \to0+$ to conclude.
\end{proof}

\begin{remark}
  The continuity of $L_K^{B_\rho}\phi$ requires only local
  regularity of $\phi$ and the absence of critical points. In particular, no further conditions on the
  behaviour of $\phi$ outside $U$ are needed, in contrast
  to the full operator $L_K\phi$. This is the key
  simplification provided the truncation
  strategy~\eqref{eq:operator-split} which lets us bypass any tail continuity condition and, together with translation invariance, any continuity assumption on the kernel. 
\end{remark}

\begin{lemma}\label{lem:parabolic-perturb-down}
Let $\Omega \subset \R^N$ be open, $p \in (1,\infty)$,
$s \in (0,1)$. Let $I \subset \R$ a bounded interval. Suppose there exists
open sets $U,V$ and a compact set $\mathcal{K} $ with $V \Subset \mathcal{K} \Subset U \Subset \Omega$ and $\rho > 0$ is such that
$B_\rho(x) \subset U$ for all $x \in \mathcal{K}$.
Suppose
$\phi \in C^2_x(U \times I) \cap C^1_t(U \times I)$
satisfies
\[
  m := \inf_{(x,t) \in U \times I}|\nabla_x\phi(x,t)|
  > 0
\]
and $\partial_t\phi + L_K^{B_\rho}\phi \ge \delta > 0$
on $V \times I$. Let $\eta \in C^2_x(U \times I) \cap C^1_t(U \times I)$ be
compactly supported in $V \times I$ with $0 \le \eta \le 1$,
and set $\phi^\varepsilon := \phi - \varepsilon\eta$. Then
there exists $0<\varepsilon_0 <1$ such that
\[
  \partial_t\phi^\varepsilon
  + L_K^{B_\rho}\phi^\varepsilon \ge \tfrac{\delta}{2}
  \quad \text{on } V \times I
\]
for every $0 < \varepsilon < \varepsilon_0$.
\end{lemma}

\begin{proof}
Let $(x,t) \in V \times I$. For $0<\varepsilon < \tfrac{1}{2}m\|\nabla_x\eta\|_{L^\infty(V
\times I)}^{-1}$ we have
\[
  |\nabla_x\phi^\varepsilon(x,t)|
  = |\nabla_x\phi(x,t) - \varepsilon\nabla_x\eta(x,t)|
  \ge m - \varepsilon\|\nabla_x\eta\|_{L^\infty}
  > \tfrac{m}{2} > 0;
\]
so Lemma~\ref{lem:pv-exists} applies to
$\phi^\varepsilon$ on $V \times I$. For the time
derivative, choosing
$\varepsilon \le \tfrac{1}{100}\delta\|\partial_t\eta\|_{L^\infty(U
\times I)}^{-1}$ gives
\[
  |\partial_t\phi^\varepsilon(x,t) -  \partial_t\phi(x,t)| \leq
   \varepsilon\|\partial_t\eta\|_{L^\infty}
  \leq \tfrac{\delta}{100}.
\]
For the spatial operator, fix $\sigma > 0$ small enough
that, by Lemma~\ref{lem:pv-exists},
\[
  |L_K^{B_\sigma}\phi(x,t)| \le \tfrac{\delta}{100}
  \quad \text{and} \quad
  |L_K^{B_\sigma}\phi^\varepsilon(x,t)|
  \le \tfrac{\delta}{100}
\]
 independent of $\varepsilon$ and $\rho$.  For the annular part
$L_K^{B_\rho \setminus B_\sigma}$, we write
\begin{align*}
  &L_K^{B_\rho \setminus B_\sigma}
    \phi^\varepsilon(x,t)
  - L_K^{B_\rho \setminus B_\sigma}\phi(x,t) \\
  &\quad = \int_{B_\rho(0) \setminus B_\sigma(0)}
    \bigl[J_p(\phi^\varepsilon(x,t)
           -\phi^\varepsilon(x+z,t))
    - J_p(\phi(x,t)-\phi(x+z,t))\bigr]
    K(z)\,dz.
\end{align*}
 Using $|J_p(a) - J_p(b)| \le c_p(|a-b| + |b|)^{p-2}|a-b|$ we get:
\begin{itemize}
  \item \emph{for $p \ge 2$:}
    $|J_p(\phi^\varepsilon(x,t)
           -\phi^\varepsilon(x+z,t))
    - J_p(\phi(x,t)-\phi(x+z,t))|
    \le C_1\varepsilon|z|^{p-1}$ for $ \varepsilon \in (0,1]$,
  \item \emph{for $1 < p < 2$:}
    $|J_p(\phi^\varepsilon(x,t)
           -\phi^\varepsilon(x+z,t))
    - J_p(\phi(x,t)-\phi(x+z,t))|
    \le C_2\varepsilon^{p-1}|z|^{p-1}$;
\end{itemize}
where $C_1$ is a constant depending on $p,|\phi|_{C^1}$ and $|\eta|_{C^1}$ and $C_2$ is a constant depending on $p$ and $ |\eta|_{C^1}$  In either case, integrating against
$K(z) \le \Lambda|z|^{-N-sp}$ over the annular part yields
\[
|L_K^{B_\rho \setminus B_\sigma}
    \phi^\varepsilon(x,t)
  - L_K^{B_\rho \setminus B_\sigma}\phi(x,t) | \leq C_3\max\{\varepsilon, \varepsilon^{p-1}\}\sigma^{p-1-N-sp}\rho^{N}
\]
for $ \varepsilon \in (0,1]$ and a constant $C_3$ independent of $\varepsilon$, $\sigma$ and $\rho$. In particular, there is a $\varepsilon'$ depending on $C_3$, $\rho$ and $\sigma$ such that for all $0 < \varepsilon < \varepsilon'$,
\[
  \bigl|L_K^{B_\rho \setminus B_\sigma}
    \phi^\varepsilon(x,t)
  - L_K^{B_\rho \setminus B_\sigma}\phi(x,t)\bigr|
  \le \tfrac{\delta}{100}.
\]
Setting
$\varepsilon_0 := \min\bigl\{
  \tfrac{m}{2\|\nabla_x\eta\|_{L^\infty}},\;
  \tfrac{\delta}{100\|\partial_t\eta\|_{L^\infty}},\;
  \varepsilon',\; 1\bigr\}$
and combining all the estimates above the conclusion follows. 
\end{proof}

\begin{lemma}\label{lem:parabolic-L310}
Let $\Omega \subset \R^N$ be open, $p \in (1,\infty)$,
$s \in (0,1)$. Let $I \subset \R$ a bounded interval. Suppose there exists
open sets $U,V$ and a compact set $\mathcal{K} $ with $V \Subset \mathcal{K} \Subset U \Subset \Omega$ and $\rho > 0$ is such that
$B_\rho(x) \subset U$ for all $x \in \mathcal{K}$.
Suppose
$\phi \in C^2_x(U \times I) \cap C^1_t(U \times I)$
satisfies $\inf_{U \times I}|\nabla_x\phi| > 0$ and
$\partial_t\phi + L_K^{B_\rho}\phi \ge 0$ pointwise in
$V \times I$. Then $\phi$ is a weak supersolution to
$\partial_t w + L_K^{B_\rho} w = 0$ in $V \times I$.
\end{lemma}

\begin{proof}
Fix $\varphi \in C^\infty_c(V \times I)$ with
$\varphi \ge 0$. For $\sigma > 0$, define the truncated
operator
\[
  L_{K,\sigma}^{B_\rho}\phi(x,t)
  := \int_{B_\rho(x) \setminus B_\sigma(x)}
     J_p(\phi(x,t)-\phi(y,t))\,K(x-y)\,dy.
\]
Since $\nabla_x\phi \neq 0$ on $U \times I$, the
principal value $L_K^{B_\rho}\phi(x,t)$ exists for every
$(x,t) \in V \times I$ by Lemma~\ref{lem:pv-exists}.
In particular, there exists $\delta_\sigma(x,t) \ge 0$ with
$\delta_\sigma(x,t) \to 0$ uniformly on $\mathrm{supp}\,\varphi$
as $\sigma \to 0$, such that
\[
  L_{K,\sigma}^{B_\rho}\phi(x,t)
  \ge L_K^{B_\rho}\phi(x,t) - \delta_\sigma(x,t)
  \quad \text{on } \mathrm{supp}\,\varphi.
\]
Using $\partial_t\phi + L_K^{B_\rho}\phi \ge 0$,
multiplying by $\varphi \ge 0$, and integrating over
$V \times I$:
\begin{align*}
  0 &\le \int_I\int_V
    (\partial_t\phi + L_K^{B_\rho}\phi)\,\varphi\,dx\,dt \\
  &\le \int_I\int_V \partial_t\phi\,\varphi\,dx\,dt
    + \int_I\int_V L_{K,\sigma}^{B_\rho}\phi\,\varphi\,dx\,dt
    + \int_I\int_V \delta_\sigma\,\varphi\,dx\,dt.
\end{align*}
For the time term, we integrate by parts. For the space term, the symmetry of $K$ and $\varphi \in C^\infty_c(V \times I)$ yields
\begin{align*}
  \int_I\int_{\R^N} &L_{K,\sigma}^{B_\rho}\phi\,\varphi\,dx\,dt \\
  &= \iint_{I \times \R^N} \int_{B_\rho(x) \setminus B_\sigma(x)}
    J_p(\phi(x,t)-\phi(y,t))\,\varphi(x,t)\,K(x-y)\,dy\,dx\,dt \\
  &= \frac{1}{2}\int_I\int_{\R^N} \int_{B_\rho(x) \setminus B_\sigma(x)}
    J_p(\phi(x,t)-\phi(y,t)) \\
  &\qquad\qquad\qquad\qquad\times
    (\varphi(x,t)-\varphi(y,t))\,K(x-y)\,dy\,dx\,dt.
\end{align*}
We note that the term $(\varphi(x,t)-\varphi(y,t))$ is non-zero only if either $x \in V$ or $y \in V$. Since the domain of integration restricts us to $|x-y| < \rho$, the entire integrand vanishes outside the region $x,y \in V + B_\rho \subset U$. Finally, as $\sigma \to 0$, the right-hand side converges to the integral over $B_\rho(x)$ by dominated convergence. 

\end{proof}
\begin{lemma}[Weak comparison principle]
\label{lem:parabolic-comparison}
Let $\Omega \subset \R^N$ be open, $p \in (1,\infty)$,
$s \in (0,1)$. Let $I = (t_1,t_2]\subset \R$ a bounded interval. Suppose there exists
open sets $U,V$ and a compact set $\mathcal{K}$ with $V \Subset \mathcal{K} \Subset U \Subset \Omega$ and $\rho > 0$ is such that
$B_\rho(x) \subset U$ for all $x \in \mathcal{K}$.
Let $u$ be a local weak supersolution and $v$ a local
weak subsolution to
$\partial_t w + L_K^{B_\rho} w = f$ in
$U \times I$, where $f \in L^\infty(V \times I)$. If
$u \ge v$ a.e.\ in
$(\R^N \setminus V) \times I$ and
$u(\cdot,t_1) \ge v(\cdot,t_1)$ a.e.\ in $V$,
then $u \ge v$ a.e.\ in $V\times I$.
\end{lemma}

\begin{proof}
The proof follows the same lines
as~\cite{KorvenpaaKuusiPalatucci2017Perron} for the spatial part while the right-hand side $f \in L^\infty$
enters as a fixed bounded perturbation which cancels out; the only difference being 
that we need more regularity in time than is available from the definition of local weak solutions. 
The standard way to do this is to use Steklov averages and we refer to \cite[Appendix B]{
Liao2024HolderParabolicFracPLap} for further details. 
\end{proof}
\begin{remark}
    We note that for the truncated operator we do not need the comparison to hold in the full complement and a $\rho-$neighbourhood of $\mathcal{K}$ is enough.  
\end{remark}

\begin{proposition}[Touching lemma]
\label{prop:weak-to-visc-parabolic-noncritical}
Let $\Omega \subset \R^N$ be open, $I \subset \R$ an open interval, $p \in (1,\infty)$, $s \in (0,1)$. Let $u$ be a local weak solution to
$\partial_t u + L_K u = 0$ in $\Omega \times I$. 
Let $U \Subset \Omega$ be open, $\rho > 0$ with
$B_\rho(x) \subset \Omega$ for all $x \in U$, and
$(x_0,t_0) \in U \times I$. Suppose
$\phi \in C^2_x(U \times I) \cap C^1_t(U \times I)$
satisfies
\[
  \phi(x_0,t_0) = u(x_0,t_0), \quad
  \phi(x,t) \ge u(x,t) \text{ for all }
  (x,t) \in U \times I, \quad
  \nabla_x\phi(x_0,t_0) \neq 0.
\]
Define
\[
  \psi(x,t) :=
  \begin{cases}
    \phi(x,t) & x \in U, \\
    u(x,t)   & x \notin U.
  \end{cases}
\]
Set $f(x,t) := - L_K^{B_\rho^c}u(x,t)$
and let $C_0$ be such that $\|f\|_{L^\infty} \le C_0$. Then
$\partial_t\psi(x_0,t_0) + L_K^{B_\rho}\psi(x_0,t_0)$
is well-defined and
\[
  \partial_t\psi(x_0,t_0)
  + L_K^{B_\rho}\psi(x_0,t_0) \le C_0.
\]
\end{proposition}

\begin{proof}
\emph{Well-defined.} Since
$\nabla_x\phi(x_0,t_0) \neq 0$ and
$\phi \in C^2_x(U \times I)$ there exists an 
$r_0 > 0$ such that
$B_{2r_0}(x_0) \Subset U$, $(t_0-2r_0,t_0+2r_0) = J_0 \Subset I$  and
$\inf_{B_{2r_0}(x_0) \times J_0}|\nabla_x\phi| > 0$.
In particular $\psi \in C^2_x(B_{2r_0}(x_0) \times J_0)
\cap C^1_t(B_{2r_0}(x_0) \times J_0)$ with
$\inf|\nabla_x\psi| > 0$ on this cylinder.
It now follows from Lemma~\ref{lem:pv-exists} and $\phi \in  C^1_t(U \times I)$ that 
$\partial_t\psi(x_0,t_0) +L_K^{B_\rho}\psi(x_0,t_0)$ is well-defined. 

\emph{Subsolution inequality.} Since $u$ satisfies
$\partial_t u + L_K u = 0$ weakly, we have
$\partial_t u + L_K^{B_\rho} u = f$ weakly and in particular, $\partial_t u + L_K^{B_\rho} u \leq C_0$ weakly. 
Suppose that
\[
  \partial_t\psi(x_0,t_0)
  + L_K^{B_\rho}\psi(x_0,t_0) > C_0.
\]
From $ \nabla_x\psi(x_0,t_0) \neq 0$ and Lemma~\ref{lem:parabolic-full-operator-continuity},
$\partial_t\psi + L_K^{B_\rho}\psi$ is continuous at
$(x_0,t_0)$ and so there
exist $r \in (0,r_0)$ and $\delta > 0$ such that on
$Q_r := B_r(x_0) \times (t_0-r, t_0+r)
\Subset U \times I$,
\[
  \inf_{Q_r}|\nabla_x\psi| > 0
  \qquad \text{and} \qquad
  \partial_t\psi + L_K^{B_\rho}\psi
  \ge C_0 + \delta \quad \text{on } Q_r.
\]
Let $\eta \in C^2_x \cap C^1_t$ be non-negative,
compactly supported in $Q_r$, with $\eta(x_0,t_0) = 1$
and $\eta(\cdot,t_0-r) \equiv 0$ in $B_r(x_0)$.
Set $\psi_\varepsilon := \psi - \varepsilon\eta$.
From Lemma~\ref{lem:parabolic-perturb-down}, by considering $\psi-C_0t$, there exists
$\varepsilon_0 > 0$ such that for all
$0 < \varepsilon < \varepsilon_0$,
\[
  \partial_t\psi_\varepsilon
  + L_K^{B_\rho}\psi_\varepsilon
  \ge C_0 + \tfrac{\delta}{2} > C_0
  \quad \text{in } Q_r.
\]
By Lemma~\ref{lem:parabolic-L310},
$\psi_\varepsilon$ is a weak supersolution to
$\partial_t w + L_K^{B_\rho} w = C_0$ in $Q_r$.

We compare $\psi_\varepsilon$ with $u$ on the parabolic
boundary of $Q_r$. Since $\eta$ is supported in $Q_r$,
we have $\psi_\varepsilon = \psi \ge u$ a.e.\ in
$(\R^N \setminus B_r(x_0)) \times (t_0-r,t_0+r)$.
At the initial time $t = t_0-r$,
$\eta(\cdot,t_0-r) \equiv 0$, so
$\psi_\varepsilon(\cdot,t_0-r)
= \phi(\cdot,t_0-r) \ge u(\cdot,t_0-r)$ in $B_r(x_0)$.
Thus, from~\cref{lem:parabolic-comparison} we get
$\psi_\varepsilon \ge u$ a.e.\ in $Q_r$. But
\[
  \psi_\varepsilon(x_0,t_0)
  = \phi(x_0,t_0) - \varepsilon\eta(x_0,t_0)
  = u(x_0,t_0) - \varepsilon < u(x_0,t_0),
\]
a contradiction.
\end{proof}

\begin{remark}\label{rem:weak-to-visc-local-min}
If $u - \phi$ has a local minimum at $(x_0,t_0)$ with
$\nabla_x\phi(x_0,t_0) \neq 0$, an analogous argument
gives
$\partial_t\psi(x_0,t_0)
+ L_K^{B_\rho}\psi(x_0,t_0) \ge -C_0$.
\end{remark}


\section{Scaling and Normalisation}
\label{sec:scaling-reduction}

The proof of Theorem~\ref{thm:main} reduces, via a change
of variables, to establishing spatial Lipschitz regularity
for a normalised solution on the fixed cylinder
$B_2 \times (-2,0]$. The precise form of the normalisation
depends on $p$. The reason for
this dichotomy is the factor $\mu^{2-p}$ appearing in the amplitude scaling
(Lemma~\ref{lem:basic-scalings}(2)): for $p \ge 2$ this factor is at most $1$, so large
amplitudes enlarge the time scale and the natural normalisation absorbs $M_R$ into a time
shift; for $1 < p < 2$ the factor exceeds $1$, so large amplitudes instead shrink the
spatial scale, and the normalisation absorbs $M_R$ into a spatial rescaling. In either case, we 
obtain a solution on $B_2 \times (-2,0]$ with its $L^\infty$ norm and tail bounded by
$1$, to which the Ishii--Lions argument of Section~\ref{sec:proofs} applies.

\subsection*{Notation}

Given $\theta > 0$, the rescaled kernel is
$K_\theta(z) := \theta^{N+sp}\, K(\theta z)$,
which satisfies~\eqref{eq:kernel-bounds} with the same
constants $\lambda, \Lambda$.

\subsection*{Scaling identities}

\begin{lemma}[Basic scalings]\label{lem:basic-scalings}
Let $u$ solve $\partial_t u + L_K u = 0$. Then:
\begin{enumerate}
\item For $\theta > 0$:
  $u_\theta(x,t) := u(\theta x, \theta^{sp} t)$
  solves
  $\partial_t u_\theta + L_{K_\theta} u_\theta = 0$.

\item For $\mu > 0$:
  $u^\mu(x,t) := \mu\, u(x,t)$
  solves
  $\partial_t u^\mu + \mu^{2-p} L_K u^\mu = 0$.

\item For $\theta, \mu > 0$:
  $u_{\theta,\mu}(x,t)
  := \mu\, u(\theta x,\; \mu^{p-2}\theta^{sp} t)$
  solves
  $\partial_t u_{\theta,\mu}
  + L_{K_\theta} u_{\theta,\mu} = 0$.

\item For $\theta, \mu > 0$:
  $\widetilde{u}_{\theta,\mu}(x,t)
  := \mu^{-1}
     u(\theta\mu^{(p-2)/(sp)} x,\; \theta^{sp} t)$
  solves
  $\partial_t \widetilde{u}_{\theta,\mu}
  + L_{K_{\theta\mu^{(p-2)/(sp)}}}
    \widetilde{u}_{\theta,\mu} = 0$.
\end{enumerate}
\end{lemma}

\subsection*{Normalisation for $p \ge 2$}

\begin{lemma}\label{lem:normalize-pge2}
Let $p \ge 2$ and let $u$ solve $\partial_t u + L_K u = 0$
in $B_{2R} \times (-2R^{sp}, 0]$. With $M_R$ as
in~\eqref{eq:MR-def} and
$\tau \in [-R^{sp}(1-M_R^{2-p}),\,0]$, the function
\[
  v(x,t) := M_R^{-1}\,u\!\left(
    Rx,\; M_R^{2-p}R^{sp}t + \tau\right)
\]
solves $\partial_t v + L_{K'} v = 0$ in $B_2 \times (-2,0]$
where $K'(z) = R^{N+sp}K(Rz)$ and satisfies
$\|v\|_{L^\infty} + \mathrm{Tail}_\infty(v;\,0,1,[-1,0])
\le 1$.
\end{lemma}

\begin{proof}
See~\cite[p.~19]{BrascoLindgrenStromqvist2021}
and~\cite[p.~34]{Tavakoli2024Perturbative}.
\end{proof}

\begin{corollary}\label{cor:holder-transfer-pge2}
Under the setting of Lemma~\ref{lem:normalize-pge2}, if
$\sup_t [v(\cdot,t)]_{C^{0,\delta}(B_{1/2})} \le A$
then \[\sup_t [u(\cdot,t)]_{C^{0,\delta}(B_{R/2})}
\le A R^{-\delta} M_R.\]
\end{corollary}

\begin{proof}
See~\cite[p.~19]{BrascoLindgrenStromqvist2021}.
\end{proof}

\subsection*{Normalisation for $1 < p < 2$}

\begin{lemma}\label{lem:normalize-plt2}
Let $1 < p < 2$ and let $u$ solve $\partial_t u + L_K u = 0$
in $B_{2R} \times (-2R^{sp},0]$. With $M_R$
as in~\eqref{eq:MR-def}
and $y_0 \in \R^N$ satisfying
$\tfrac{1}{2}R M_R^{(p-2)/(sp)} + |y_0| \le \tfrac{1}{2}R$,
the function
\[
  v(x,t)
  := M_R^{-1}\,u\!\left(
       R M_R^{(p-2)/(sp)}x + y_0,\; R^{sp}t\right)
\]
solves $\partial_t v + L_{K'} v = 0$ in $B_2 \times (-2,0]$ where $K'(z) :=
  \gamma^{\!N+sp}
  K\!\left(\gamma\,z\right)$ for $\gamma = R M_R^{\frac{p-2}{sp}}$
and satisfies
$\|v\|_{L^\infty}
+ \mathrm{Tail}_\infty(v;\,0,1,[-1,0]) \le 1$.
\end{lemma}

\begin{proof}
See~\cite[p.~19]{GarainLindgrenTavakoli2025}.
\end{proof}

\begin{corollary}\label{cor:holder-transfer-plt2}
Under the setting of Lemma~\ref{lem:normalize-plt2}, if
$\sup_t [v(\cdot,t)]_{C^{0,\delta}(B_{1/2})} \le A$
then
\[\sup_t [u(\cdot,t)]_{C^{0,\delta}(B_{R/2})}
\le A R^{-\delta} M_R^{1+(2-p)\delta/(sp)}.\]
\end{corollary}

\begin{proof}
See~\cite[p.~19]{GarainLindgrenTavakoli2025}.
\end{proof}

\section{Proof of Main Theorem}
\label{sec:proofs}

We prove spatial Lipschitz regularity for the unknown $u$ normalised as in Section~\ref{sec:scaling-reduction}. Throughout the section, $p > 1$, $sp > p-1$, and $u$ is a continuous local weak solution to
\begin{equation}\label{eq:parabolic-homogeneous}
  \partial_t u + L_K u = 0 \quad \text{in } B_2 \times (-2, 0],
\end{equation}
satisfying the normalisation bound
\begin{equation}\label{eq:normalization-bound}
  \|u\|_{L^\infty(B_2 \times (-2,0])} + \mathrm{Tail}_\infty(u;\,0,2,\,(-2,0]) \le 1.
\end{equation}

Following the truncation strategy outlined in Section~\ref{sec:touching}, we decompose the operator as $L_K u = L_K^{B_1} u + L_K^{B_1^c} u$. Setting $f(x,t) := -L_K^{B_1^c} u(x,t)$, the normalisation bounds imply that $f \in L^\infty(B_1 \times (-1,0])$ with $\|f\|_{L^\infty} \le C_0$ for a universal constant $C_0 > 0$. Hence, $u$ acts as a weak subsolution (and supersolution) to the truncated problem:
\begin{align}
  \partial_t u + L_K^{B_1} u &\le C_0 \quad \text{in } B_1 \times (-1,0], \label{eq:trunc-sub} \\
  \partial_t u + L_K^{B_1} u &\ge -C_0 \quad \text{in } B_1 \times (-1,0]. \label{eq:trunc-super}
\end{align}

Under the above setup, the arguments now largely follow along the lines of \cite{BiswasTopp2025} with the difference that
we use a parametrized family of Ishii-Lions functionals to ensure that $|a_{\nu}| \neq 0$ (cf.~\cref{fig:domain_decomposition}) and also since we have truncated the operator, some of the computations simplify and in particular, the final tail integral in \cite{BiswasTopp2025} has already been taken care of in the form of $C_0$ above.  

\subsection*{Penalisation profiles}
We work with two penalisation profiles. For the H\"{o}lder bootstrap we use:
\begin{equation}\label{eq:holder-profile}
  \varphi_\gamma(r) := r^\gamma, \quad \gamma \in (0,1).
\end{equation}
For the Lipschitz step we work with the almost-linear profile:
\begin{equation}\label{eq:lip-profile}
  \tilde\varphi(r) := r + \frac{r}{\log r}, \quad r \in (0, r_\circ],
\end{equation}
for $r_\circ \in (0, e^{-2})$ small enough, with $\tilde\varphi(0) := 0$. On $(0, r_\circ]$, this profile satisfies
\begin{equation}\label{eq:lip-profile-props}
  \tfrac{r}{2} \le \tilde\varphi(r) \le r, \quad
  \tfrac{1}{2} \le \tilde\varphi'(r) \le 1, \quad
  -\tfrac{2}{r\log^2 r} \le \tilde\varphi''(r) \le -\tfrac{1}{2r\log^2 r} < 0.
\end{equation}
The strict concavity $\tilde\varphi'' < 0$ ensures the cone estimates for the integral bounds (cf.~\cref{fig:domain_decomposition}) still holds, which would fail for the purely linear profile $\varphi(r) = r$.

\subsection{The Ishii--Lions framework}
\label{sec:parabolic-ishii-lions}

Fix radii $1 \le \varrho_1 < \varrho_2 \le 3/2$, and set $\tilde\varrho := \frac{1}{4}(\varrho_2 - \varrho_1) < 1$. Let $m \ge 3$ and choose $m_1 > 0$ such that $m_1\psi(x) > 2$ for all $|x| \ge (\varrho_1+\varrho_2)/2$, where $\psi(x) := ((|x|^2 - \varrho_1^2)_+)^m$. Let $\varphi$ be either penalisation profile, set $\phi_L(x,y) := L\varphi(|x-y|)$, and consider the spatial gap function:
\begin{equation}\label{eq:contradiction-assumption}
     g(x,y,t) := u(x,t) - u(y,t) - \phi_L(x,y) - m_1\psi(x).
\end{equation}
If $g \le 0$ in $\overline{B}_2 \times \overline{B}_2 \times [-1,0]$ for all sufficiently large $L$, then the conclusion follows. Suppose for contradiction that for arbitrarily large $L$, there exists $(x_L,y_L,t_L) \in B_2 \times B_2 \times(-1,0)$ such that $g(x_L,y_L,t_L) = 2\alpha_L > 0$. Fix an open interval $J \Subset (-2,0)$ with $t_L \in J$ such that $\sup_{x,y}g(x,y,t) > \alpha_L$ for all $t \in J$. 

We define the time-doubled functional on $\overline{B}_2 \times \overline{B}_2 \times \bar{J} \times \bar{J}$:
\begin{equation}\label{eq:Phi-nu-def}
  \Phi_\nu(x,y,t,\tau) := u(x,t) - u(y,\tau) - \phi_L(x,y) - m_1\psi(x) - \frac{(t-\tau)^2}{2\nu}.
\end{equation}
Let $(x_\nu,y_\nu,t_\nu,\tau_\nu)$ be a global maximiser of $\Phi_\nu$, and denote $M_\nu := \Phi_\nu(x_\nu,y_\nu,t_\nu,\tau_\nu)$. Set $a_\nu := x_\nu - y_\nu$ and $b_\nu := t_\nu - \tau_\nu$. Since $M_\nu \ge \Phi_\nu(x_L, y_L, t_L, t_L) \ge 2\alpha_L > 0$, standard arguments yield that as $\nu \to 0$, $b_\nu \to 0$ and $\liminf_{\nu \to 0} |a_\nu| \ge \alpha_0 > 0$. 
Furthermore, the penalty $m_1\psi(x)$ constrains the spatial maximisers to $B_{\varrho_2} \subset B_{3/2}$.

Consequently, we may fix a $\nu > 0$ sufficiently small such that $t_\nu, \tau_\nu \in (-2,0)$ and $|a_\nu| > 0$. Since $M_\nu \ge 2\alpha_L > 0$ and $\|u\|_{L^\infty} \le 1$, we have $L\varphi(|a_\nu|) \le 2$, which provides the uniform upper bound $|a_\nu| \le \varphi^{-1}(2/L)$. In particular, $|a_\nu| \to 0$ as $L \to \infty$.

\subsection{The master inequality}
\label{sec:master-inequality}

To apply the Touching Lemma (Proposition~\ref{prop:weak-to-visc-parabolic-noncritical}) at the points $(x_\nu,t_\nu)$ and $(y_\nu,\tau_\nu)$, we consider the following test functions that globally bound $u$ from above and below. Define:
\begin{align*}
  w_1(x,t) &:= \begin{cases}
     L\varphi(|x-y_\nu|) + m_1\psi(x) + u(y_\nu,\tau_\nu) + \frac{(t-\tau_\nu)^2}{2\nu} + M_\nu,& \text{if  $x\in B_{\delta_\nu}(x_{\nu})$}\\
      u(x,t) & \text{otherwise}
    \end{cases}  \\
  w_2(y,\tau) &:= \begin{cases}
     u(x_\nu,t_\nu) - L\varphi(|x_\nu-y|) - m_1\psi(x_\nu) - \frac{(t_\nu-\tau)^2}{2\nu} - M_\nu,& \text{if  $y\in B_{\delta_\nu}(y_{\nu})$}\\
      u(y,\tau) & \text{otherwise}
    \end{cases} 
\end{align*}
where $\delta_\nu := \varepsilon_1|a_\nu|$ for a small geometric constant $\varepsilon_1 \in (0,1/2)$ fixed as in~\cite{BiswasTopp2025}. By construction, we have $w_1(x,t) \ge u(x,t)$, $w_2(y,\tau) \le u(y,\tau)$, $w_1(x_\nu,t_\nu)=u(x_\nu,t_\nu)$ and $w_2(y_\nu,\tau_\nu)=u(y_\nu,\tau_\nu)$. Furthermore, since $|a_\nu| > 0$ and $\varphi'$ is strictly positive, $\nabla_x w_1(x_\nu,t_\nu) \neq 0$ and $\nabla_y w_2(y_\nu,\tau_\nu) \neq 0$. 

We note that the choice of $m_1\psi$ and taking $L$ large forces $x_\nu, y_\nu \in B_{3/2}$ as in \cite{BiswasTopp2025}. Thus, from  Proposition~\ref{prop:weak-to-visc-parabolic-noncritical} we get:
\[
  \frac{b_\nu}{\nu} + L_K^{B_1} w_1(x_\nu,t_\nu) \le C_0,
  \qquad
  \frac{b_\nu}{\nu} + L_K^{B_1} w_2(y_\nu,\tau_\nu) \ge -C_0.
\]
The term $b_\nu/\nu$ is the same in both inequalities and cancels upon subtraction, yielding the master inequality:
\begin{equation}\label{eq:parabolic-master}
  L_K^{B_1} w_1(x_\nu,t_\nu) - L_K^{B_1} w_2(y_\nu,\tau_\nu) \le 2C_0.
\end{equation}

Let $\delta_0 \in (0, 1/2)$ be a small geometric constant fixed as in~\cite{BiswasTopp2025}. We decompose the left-hand side of~\eqref{eq:parabolic-master} into integrals over four disjoint regions of $B_1(0)$. For the integration variable $z \in B_1(0)$, set:
\begin{align*}
  C &:= \{z \in B_{\delta_0 |a_\nu|} : |\langle \hat{a}_\nu, z \rangle| \ge (1-\delta_0)|z|\}, \\
  D_1 &:= B_{\delta_\nu} \setminus C, \\
  D_2 &:= B_{\tilde\varrho} \setminus B_{\delta_\nu}, \\
  D_3 &:= B_1 \setminus B_{\tilde\varrho}.
\end{align*}
\begin{remark}
For the Hölder profile, $\delta_0$ is a fixed geometric constant. For the Lipschitz profile, following~\cite[Lemma~2.2(ii)]{BiswasTopp2025}, the cone $C$ is defined with $\delta_0 = \delta_1(\log^2|a_\nu|)^{-1}$ for a small constant $\delta_1 > 0$; in particular, $\delta_0$ depends on $|a_\nu|$ in this case.
\end{remark}
For a function $w : \R^N \times \R \to \R$ and $(x,t,z) \in \R^N \times \R \times \R^N$, we write
\[
  \Theta_1 w(x,t,z) := w(x,t) - w(x+z,t)
\]
for the spatial first-order increment. We write $I_i$ ($i=1,2,3,4$) for the contribution to the left-hand side of~\eqref{eq:parabolic-master} from each region:
\[
  I_i := \int_{R_i} \bigl[J_p(\Theta_1 w_1(x_\nu,t_\nu,z)) - J_p(\Theta_1 w_2(y_\nu,\tau_\nu,z))\bigr] K(z)\,dz,
\]
where $R_1 = C$, $R_2 = D_1$, $R_3 = D_2$, $R_4 = D_3$.

\begin{figure}[htpb]
    \centering
    \begin{tikzpicture}[scale=1.1]
        
        \def\Rthree{4.5} 
        \def\Rtwo{3.0}   
        \def\Rone{1.5}   
        \def\coneangle{35} 

        \draw[fill=black!5, draw=black, thick] (0,0) circle (\Rthree);
        
        \draw[fill=black!15, draw=black, thick, dashed] (0,0) circle (\Rtwo);
        
        \draw[fill=black!30, draw=black, thick, dotted] (0,0) circle (\Rone);
        
        \draw[fill=black!50, draw=black, thick] (0,0) -- (\coneangle:\Rone) arc (\coneangle:-\coneangle:\Rone) -- cycle;
        \draw[fill=black!50, draw=black, thick] (0,0) -- (180-\coneangle:\Rone) arc (180-\coneangle:180+\coneangle:\Rone) -- cycle;

        \draw[gray, dashed] (180+\coneangle:\Rone) -- (\coneangle:\Rone);
        \draw[gray, dashed] (180-\coneangle:\Rone) -- (-\coneangle:\Rone);

        \draw[->, >=stealth, thick] (0,0) -- (2.5,0) node[anchor=north west] {$\hat{a}_\nu$};
        
        \filldraw (0,0) circle (1.5pt) node[anchor=north west] {$0$};

        
        \node[text=white, font=\bfseries] at (0.9, 0) {$C$};
        \node[text=white, font=\bfseries] at (-0.9, 0) {$C$};
        
        \node[font=\bfseries] at (0, 0.9) {$D_1$};
        \node[font=\bfseries] at (0, -0.9) {$D_1$};
        
        \node[font=\bfseries] at (0, 2.25) {$D_2$};
        
        \node[font=\bfseries] at (0, 3.75) {$D_3$};

        \node[anchor=south east, inner sep=1pt] at (135:\Rone) {$|z|=\delta_\nu$};
        \node[anchor=south east, inner sep=1pt] at (135:\Rtwo) {$|z|=\tilde{\varrho}$};
        \node[anchor=south east, inner sep=1pt] at (135:\Rthree) {$|z|=1$};

    \end{tikzpicture}
    \caption{Geometric decomposition of the integration domain $B_1$ into regions $C, D_1, D_2,$ and $D_3$, oriented along the unit vector $\hat{a}_\nu = a_\nu / |a_\nu|$.}
    \label{fig:domain_decomposition}
\end{figure}
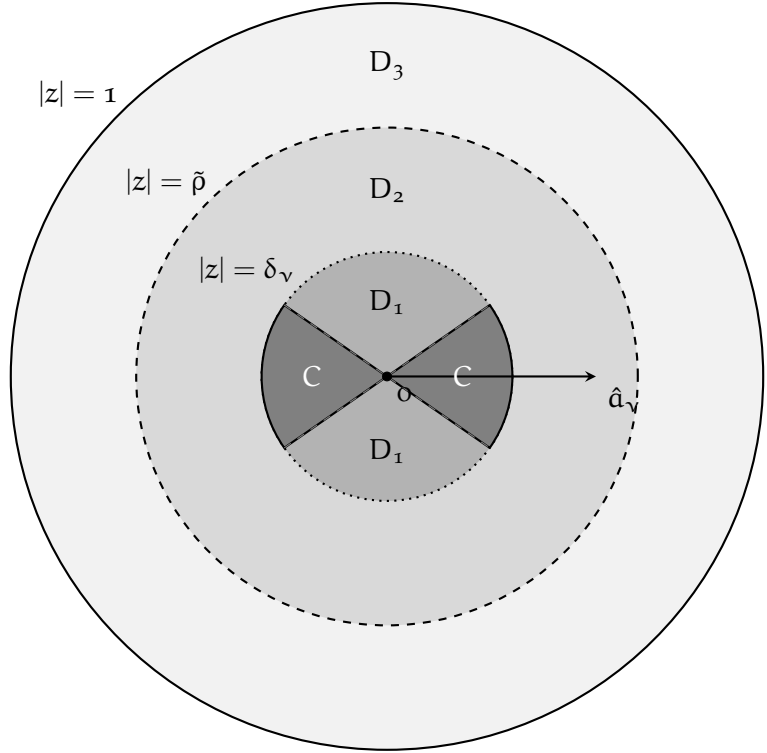

\subsection{Integral estimates}
\label{sec:integral-estimates}

\begin{lemma}
\label{lem:parabolic-I1-I2}
We have the following bounds:
\begin{itemize}
  \item \emph{H\"{o}lder profile:} $I_1 + I_2 \ge \tfrac{C_\varepsilon}{2} L^{p-1}|a_\nu|^{\gamma(p-1)-ps}$.
  \item \emph{Lipschitz profile:} $I_1 + I_2 \ge \tfrac{C_\varepsilon}{2} L^{p-1}|a_\nu|^{p-1-ps} (\log^2|a_\nu|)^{-\beta}$, where $\beta := \tfrac{N+1}{2}+p-sp$.
\end{itemize}
\end{lemma}
\begin{proof}
We note that for $z \in B_{\delta_\nu}$, the point $x_\nu + z$ lies in $B_{\delta_\nu}(x_\nu)$. We compute the spatial increment directly:
\begin{align*}
  \Theta_1 w_1(x_\nu,t_\nu,z)
  &= w_1(x_\nu,t_\nu) - w_1(x_\nu+z,t_\nu) \\
  &= \Bigl[L\varphi(|a_\nu|) + m_1\psi(x_\nu) + u(y_\nu,\tau_\nu) + \tfrac{b_\nu^2}{2\nu} + M_\nu\Bigr] \\
  &\quad - \Bigl[L\varphi(|x_\nu+z-y_\nu|) + m_1\psi(x_\nu+z) + u(y_\nu,\tau_\nu) + \tfrac{b_\nu^2}{2\nu} + M_\nu\Bigr] \\
  &= L[\varphi(|a_\nu|) - \varphi(|a_\nu+z|)] + m_1[\psi(x_\nu) - \psi(x_\nu+z)],
\end{align*}
to see that the time-dependent terms $u(y_\nu,\tau_\nu)$, $(t_\nu-\tau_\nu)^2/(2\nu)$, and $M_\nu$ cancel exactly. Hence $\Theta_1 w_1(x_\nu,t_\nu,z)$ is purely spatial and coincides with the elliptic increment from~\cite{BiswasTopp2025}. An identical computation shows the same for $\Theta_1 w_2(y_\nu,\tau_\nu,z)$ on $B_{\delta_\nu}$. The lower bounds for the H\"{o}lder profile then follow directly from~\cite[Lemmas~3.1, 3.2, 4.1, 4.2]{BiswasTopp2025}, and for the Lipschitz profile from~\cite[Proof of Theorem~2.1]{BiswasTopp2025}.
\end{proof}

\begin{lemma}
\label{lem:parabolic-I3-super}
Let $p \ge 2$ and suppose $u(\cdot,t) \in C^{0,\kappa}(B_{\varrho_2})$ uniformly in $t$. We have the following bounds:
\begin{itemize}
  \item \emph{H\"{o}lder profile:} $I_3 \ge -C_{\kappa,\varepsilon_1} |a_\nu|^{\gamma(p-1)-sp}$.
  \item \emph{Lipschitz profile:} $I_3 \ge -C_{\kappa,\varepsilon_1}|a_\nu|^\sigma (\log^2|a_\nu|)^{sp-\kappa(p-2)-1}$, where $\sigma := \kappa(p-2)-sp+1+\tfrac{m-1}{m}\kappa > 0$.
\end{itemize}
The constant $C_{\kappa,\varepsilon_1}$ depends on $[u]_\kappa, \varepsilon_1, N, p, s, \Lambda, m, m_1$, but is independent of $L$.
\end{lemma}

\begin{proof}
For $z \in D_2 = B_{\tilde\varrho} \setminus B_{\delta_\nu}$, we have $|z| > \delta_\nu$, so $x_\nu + z \notin B_{\delta_\nu}(x_\nu)$ and $y_\nu + z \notin B_{\delta_\nu}(y_\nu)$. In particular, 
\[
  w_1(x_\nu+z,t_\nu) = u(x_\nu+z,t_\nu), \qquad w_2(y_\nu+z,\tau_\nu) = u(y_\nu+z,\tau_\nu).
\]
Thus $I_3 = \int_{D_2} \bigl[J_p(A) - J_p(B)\bigr]K(z)\,dz$, where $A := u(x_\nu,t_\nu) - u(x_\nu+z,t_\nu)$ and $B := u(y_\nu,\tau_\nu) - u(y_\nu+z,\tau_\nu)$.

As before, from $\Phi_\nu(x_\nu+z, y_\nu+z, t_\nu, \tau_\nu) \le \Phi_\nu(x_\nu, y_\nu, t_\nu, \tau_\nu)$ we have that the time penalty cancels exactly, and so we get a pointwise lower bound independent of time as follows:
\[
  A - B \ge \phi_L(x_\nu+z, y_\nu+z) - \phi_L(x_\nu, y_\nu) + m_1(\psi(x_\nu) - \psi(x_\nu+z)).
\]
The profile term $\phi_L$ is controlled using the concavity of $\varphi$; see~\cite[Lemma~3.3]{BiswasTopp2025}. Since $|A|, |B|$ are bounded by $[u]_\kappa |z|^\kappa$, integrating $J_p(A) - J_p(B)$ over $D_2$ reduces to bounding the purely spatial terms. 

For the H\"{o}lder profile, the evaluation of this integral across the subcases for the sign of $\kappa(p-2)+2-sp$ is carried out in~\cite[Proposition~3.5]{BiswasTopp2025}. For the Lipschitz profile, the corresponding logarithmic bounds are evaluated in~\cite[Proof of Theorem~2.1]{BiswasTopp2025}.
\end{proof}

\begin{lemma}
\label{lem:parabolic-I3-sub}
Let $1<p<2$. We have the following bounds:
\begin{itemize}
  \item \emph{H\"{o}lder profile:} $I_3 \ge -C_{\varepsilon_1} |a_\nu|^{\gamma(p-1)-sp}$.
  \item \emph{Lipschitz profile:} $I_3 \ge -C_{\varepsilon_1} (|a_\nu|^{(p-1)-sp} (\log^{2}|a_\nu|)^{-\beta}+1)$.
\end{itemize}
The constant $C_{\varepsilon_1}$ depends on $\varepsilon_1, N, p, s, m, m_1$, but is independent of $L$.
\end{lemma}

\begin{proof}
As in the superquadratic case, the maximality of $\Phi_\nu$ yields a pointwise lower bound for $A-B$ that is independent of the time parameters $t_\nu$ and $\tau_\nu$. Since $u$ is bounded by $1$, we have $|A|, |B| \le 2$. This allows us to use the uniform algebraic bound
\[
  |J_p(A)-J_p(B)| \le c_p|A-B|^{p-1}.
\]
Integrating this bound against the kernel $K(z)$ over $D_2$ reduces to the elliptic estimates. For the H\"{o}lder profile, the resulting integration is evaluated in~\cite[Section~4.2]{BiswasTopp2025}. For the Lipschitz profile, the corresponding logarithmic bounds are established in~\cite[Proof of Theorem~2.1]{BiswasTopp2025}.
\end{proof}

\begin{lemma}
\label{lem:parabolic-I4}
The integral $I_4$ over $D_3 = B_1 \setminus B_{\tilde\varrho}$ is uniformly bounded: $|I_4| \le C_4$, where $C_4$ depends only on $N, s, p, \Lambda, \tilde\varrho$ and the normalisation bound~\eqref{eq:normalization-bound}, independent of $L$.
\end{lemma}
\begin{proof}
Since the domain $D_3$ is bounded away from the origin by the fixed geometric constant $\tilde\varrho$, the kernel $K(z)$ is non-singular. The terms inside $J_p$ are bounded absolutely by $2\|u\|_\infty \le 2$. The result follows by direct integration.
\end{proof}
\subsection{Proof of Theorem~\ref{thm:main}}
\label{sec:main-proofs}

\begin{proposition}[Spatial H\"{o}lder bootstrap]
\label{prop:parabolic-holder-bootstrap}
Let $sp > p-1$ and suppose $u(\cdot,t) \in C^{0,\kappa}(B_{\varrho_2})$ uniformly in $t$ for some
$\kappa \in [0,1)$. Then for any $\gamma < \min\!\bigl\{1,\,\kappa+\tfrac{1}{p-1}\bigr\}$ there exists
$L_\gamma > 0$ such that $\sup_t [u(\cdot,t)]_{C^{0,\gamma}(B_{\varrho_1})} \le L_\gamma$.
\end{proposition}

\begin{proof}
We use the H\"{o}lder profile $\varphi_\gamma$. Suppose for contradiction that the supremum of the
gap function~\eqref{eq:contradiction-assumption} is strictly positive for arbitrarily large $L$.
For a fixed $L$, construct $\Phi_\nu$ and fix a small $\nu > 0$ as in
Section~\ref{sec:parabolic-ishii-lions}.

Combining the master inequality~\eqref{eq:parabolic-master} with the integral bounds from
Section~\ref{sec:integral-estimates}, we obtain:

\smallskip
\noindent\textit{Case $p \ge 2$:}
$\bigl(\tfrac{C_\varepsilon}{2}L^{p-1} - C_{\kappa,\varepsilon_1}\bigr)
|a_\nu|^{\gamma(p-1)-ps} \le 2C_0 + C_4$.

\smallskip
\noindent\textit{Case $1 < p < 2$:}
$\bigl(\tfrac{C_\varepsilon}{2}L^{p-1} - C_{\varepsilon_1}\bigr)
|a_\nu|^{\gamma(p-1)-ps} \le 2C_0 + C_4$.

\smallskip
As noted in Section~\ref{sec:parabolic-ishii-lions}, $L\varphi_\gamma(|a_\nu|) \le 2$, which
implies $|a_\nu| \le (2/L)^{1/\gamma} \to 0$ as $L \to \infty$.

We now verify that $\gamma(p-1) - sp < 0$, so that $|a_\nu|^{\gamma(p-1)-sp} \to +\infty$
as $|a_\nu| \to 0$. Since $\gamma < \min\!\bigl\{1,\,\kappa+\tfrac{1}{p-1}\bigr\}$, there are
two cases:
\begin{itemize}
  \item If $\kappa + \tfrac{1}{p-1} \ge 1$, then $\gamma < 1$, so
    $\gamma(p-1) < p-1 < sp$, giving $\gamma(p-1) - sp < 0$.
  \item If $\kappa + \tfrac{1}{p-1} < 1$, then $\gamma < \kappa + \tfrac{1}{p-1}$,
    so $\gamma(p-1) < \kappa(p-1) + 1$. Since $sp > p-1$ forces
    $\tfrac{p-2}{p-1} < \tfrac{sp-1}{p-1}$, and this case requires
    $\kappa \le \tfrac{p-2}{p-1}$, we have $\kappa(p-1)+1 \le sp$, hence
    $\gamma(p-1) - sp < 0$.
\end{itemize}
Therefore $|a_\nu|^{\gamma(p-1)-sp} \to +\infty$ as $L \to \infty$. The right-hand side
$2C_0 + C_4$ is a fixed constant independent of $L$. Thus, for sufficiently large $L$, the
left-hand side strictly dominates, yielding a contradiction.
\end{proof}

For the following we choose:
\[
  m \ge \max\!\Bigl\{3,\; \Bigl\lceil \tfrac{1}{p - sp} \Bigr\rceil + 1\Bigr\}
\]

\begin{theorem}[Lipschitz regularity at normalised scale]
\label{thm:parabolic-lipschitz}
Let $p > 1$, $sp > p-1$. Then
$\sup_t[u(\cdot,t)]_{C^{0,1}(B_1)} \le C$, where $C$ depends only on $N, s, p, \lambda, \Lambda$.
\end{theorem}

\begin{proof}
We run Proposition~\ref{prop:parabolic-holder-bootstrap} iteratively, using a finite decreasing sequence of radius pairs $\varrho_1^{(k)} < \varrho_2^{(k)}$ all lying in $[1, 3/2]$, with the final pair being $\varrho_1 = 1$, $\varrho_2 = 3/2$. Starting from $\kappa = 0$ (which holds trivially since $\|u\|_{L^\infty} \le 1$), each step improves the H\"{o}lder exponent by at least $\tfrac{1}{p-1}$. After finitely many steps, we obtain $u(\cdot,t) \in C^{0,\kappa}(B_{3/2})$ uniformly in time for $\kappa$ arbitrarily close to $1$.

Fix $\kappa < 1$ sufficiently close to $1$ and $m$ as chosen above so that
\[
  \sigma := \kappa(p-2) - sp + 1 + \tfrac{m-1}{m}\kappa > 0.
\]
We employ the Lipschitz profile $\tilde\varphi$. Assuming for contradiction that the gap supremum is positive for arbitrarily large $L$, the master inequality yields:

\smallskip
\noindent\textit{Case $p \ge 2$:}
\begin{align*}
  \Bigl(\tfrac{C_\varepsilon}{2}L^{p-1}
  - C_{\kappa,\varepsilon_1}
    |a_\nu|^{\sigma}
    (\log^2|a_\nu|)^{sp-\kappa(p-2)-1}
  \Bigr)
  |a_\nu|^{p-1-sp}(\log^2|a_\nu|)^{-\beta}
  \le 2C_0 + C_4.
\end{align*}

\noindent\textit{Case $1 < p < 2$:}
\begin{align*}
  \Bigl(\tfrac{C_\varepsilon}{2}L^{p-1}
  - C_{\varepsilon_1}
  \Bigr)
  |a_\nu|^{p-1-sp}(\log^2|a_\nu|)^{-\beta}
  \le 2C_0 + C_4.
\end{align*}

\smallskip
Since $L\tilde\varphi(|a_\nu|) \le 2$ and $\tilde\varphi(r) \ge r/2$, we have $|a_\nu| \le 4/L$, which ensures $|a_\nu| \to 0$ as $L \to \infty$. Thus, the above inequalities cannot hold for $L$ large. 
\end{proof}

\begin{proof}[Proof of Theorem~\ref{thm:main}]

We apply \cref{lem:normalize-pge2} (resp.\cref{lem:normalize-plt2}) to reduce the
general equation to the normalised setting of \cref{thm:parabolic-lipschitz}. We now take
$\delta = 1$ in \cref{cor:holder-transfer-pge2} to get
\[
  \sup_t [u(\cdot,t)]_{C^{0,1}(B_{R/2})} \le C\,M_R\,R^{-1},
\]
which is~\eqref{eq:lip-pge2}. For $1 < p < 2$, taking $\delta = 1$ in
Corollary~\ref{cor:holder-transfer-plt2} gives
\[
  \sup_t [u(\cdot,t)]_{C^{0,1}(B_{R/2})}
  \le C\,M_R^{1+\frac{2-p}{sp}}\,R^{-1},
\]
which is~\eqref{eq:lip-plt2}. 
\end{proof}

\section{Extensions and Further Remarks}
\label{sec:extensions}
The Ishii--Lions method adapted to the nonlocal, nonlinear context, as observed in \cite{BiswasTopp2025}, is widely applicable as we survey below in brief. 

\subsection*{Higher H\"{o}lder regularity for
discontinuous kernels}
Using $\varphi_\gamma$ with $\gamma < \min\{1,sp/(p-1)\}$
and running the bootstrap of
Proposition~\ref{prop:parabolic-holder-bootstrap} from
$\kappa = 0$, any local weak solution
to~\eqref{eq:main-eq} belongs to $C^{0,\gamma}_x$
locally, uniformly in time. When $sp/(p-1) \ne 1$ the
endpoint $\gamma = \min\{1, sp/(p-1)\}$ is also attained,
with both cases handled by the Lipschitz-profile argument
of Section~\ref{sec:parabolic-ishii-lions}. These are
parabolic analogues of~\cite[Theorem~2.1]{BiswasTopp2025}
and extend those
of~\cite{BrascoLindgrenStromqvist2021,
GarainLindgrenTavakoli2025} to all kernels
satisfying~\eqref{eq:kernel-bounds}.

\subsection*{Equations with $L^\infty$ source}
For $\partial_t u + L_K u = f$ with
$f \in L^\infty(\Omega \times I)$, the master
inequality~\eqref{eq:parabolic-master} acquires a term
$2\|f\|_{L^\infty}$ on the right. Since this is bounded
while the left side diverges as $L \to \infty$, the same
Lipschitz conclusion holds;
cf.~\cite[Theorem~1.1]{BiswasTopp2025} and
\cite{Tavakoli2024Perturbative}.

\subsection*{Lower-order terms}
For $\partial_t u + L_K u = b(x,t,u,\nabla u)$ with
$|b| \le C_b(1+|\xi|^\alpha)$, $\alpha \in [0,1)$, the
gradient of the test function at the touching point
$x_\nu$ satisfies
$|\nabla w_1(x_\nu,t_\nu)| \le CL\varphi'(|a_\nu|)$,
so the lower-order term contributes at most
$C_b(1 + (L\varphi'(|a_\nu|))^\alpha)$ to the master
inequality. Since $\varphi'(|a_\nu|)$ grows strictly
slower than $L^{(p-1-\alpha)/(1-\alpha)}$ for $\alpha < 1$,
this is of lower order than
$L^{p-1}|a_\nu|^{p-1-sp}$ and the Lipschitz conclusion
extends. The precise admissible range of $\alpha$ is left
to future work.

\subsection*{Fractional $(p,q)$-Laplacian}
For $\partial_t u + (-\Delta_p)^{s_1}u
+ (-\Delta_q)^{s_2}u = 0$, the Ishii--Lions argument
applies to each operator separately. The Lipschitz
threshold is $\min\{s_1p/(p-1),\, s_2q/(q-1)\} > 1$,
that is, both operators must individually satisfy the
Lipschitz condition; cf.~\cite[Theorem~5.1]{BiswasTopp2025}.

\subsection*{Nonlocal double-phase operator}
For $\partial_t u + (-\Delta_p)^{s_1}u
+ \xi(x)(-\Delta_q)^{s_2}u = 0$ with
$\xi \in C^{0,\alpha}$, $s_2q \le s_1p+\alpha$, the
additional term $(\xi(x_\nu)-\xi(y_\nu))
L_q[D_2]w_2$ is controlled by $|a_\nu|^\alpha$ and is of
lower order; cf.~\cite[Theorem~5.2]{BiswasTopp2025}.

\bibliography{main}
\end{document}